\documentclass[twoside,bezier,12pt, reqno]{amsart}
\usepackage{amsmath, amsthm, amscd, amsfonts,url, amssymb,graphicx,graphics, color}
\usepackage[bookmarksnumbered, plainpages]{hyperref}
\usepackage{pgf,tikz}
\usetikzlibrary{arrows}
\newtheorem{theorem}{Theorem}[section]

\newtheorem{proposition}[theorem]{Proposition}

\theoremstyle{definition}

\theoremstyle{remark}

\numberwithin{equation}{section}
\begin{document}

\title{A note on  the automorphism groups of  Johnson graphs}
\author{S. Morteza Mirafzal}

\address{ Department of Mathematics, Lorestan University, Khoramabad, Iran}
\email{smortezamirafzal@yahoo.com}
\email{mirafzal.m@lu.ac.ir}

\begin{flushleft}
\end{flushleft} 
\begin{abstract} The Johnson graph $J(n, i)$ is defined as the graph whose vertex set is the set
of all $i$-element subsets of $\{1, . . . , n \}$, and two vertices are adjacent whenever the
cardinality of their intersection is equal to $i$-1. In Ramras and Donovan [SIAM
J. Discrete Math, 25(1): 267-270, 2011], it is proved that  if $ n \neq 2i$,  then the
  automorphism group of $J(n, i)$
is isomorphic with the group $Sym(n)$ and it is conjectured
 that if $n = 2i$, then the automorphism group of $J(n, i)$ is isomorphic with the group
 $ Sym(n) \times \mathbb{Z}_2$. In this paper,  we will find these results by different methods.
  We will prove the conjecture in the affirmative.   \
  
 \


\end{abstract}
\maketitle

\section{Introduction} 
Johnson graphs arise from the association schemes
with the same name. They are defined as follows.\

Given $n,m \in \mathbb{N}$ with $ m \leqq n-1 $,
the Johnson graph $J(n,m)$ is defined by: \

(1) The vertex set is the set of all subsets of $I = \{ 1,2, ..., n\} $ with cardinality exactly $m$.\

(2) Two vertices are adjacent if and only if the cardinality of their intersection is
equal to $m-1$.\

The Johnson graph $J(n,m)$ is a vertex transitive graph [7].
It follows from the definition that  for $m = 1$, the Johsnon graph $J(n, 1)$ is the complete graph
$K_n$. For $m = 2$ the Johnson graph $J(n, 2)$ is the line graph of the complete graph on $n$ vertices,
also known as the triangular graph $T(n)$. Thus, for instance, $J(5, 2)$ is the complement of the
Petersen graph, displayed in Figure 1, and  in general, $J(n, 2)$ is the complement of the Kneser
graph $K(n, 2)$. \

\definecolor{qqqqff}{rgb}{0.,0.,0.8.}
\begin{tikzpicture}[line cap=round,line join=round,>=triangle 45,x=1.0cm,y=1.0cm]
\clip(-2.3,-0.9) rectangle (7.32,6.3);
\draw (1.84,4.12)-- (1.44,3.18);
\draw (1.44,3.18)-- (1.96,1.84);
\draw (1.84,4.12)-- (2.76,5.);
\draw (2.76,5.)-- (4.48,4.98);
\draw (4.48,4.98)-- (5.7,3.94);
\draw (5.7,3.94)-- (5.8,2.88);
\draw (5.8,2.88)-- (5.4,1.76);
\draw (5.4,1.76)-- (4.56,1.02);
\draw (4.56,1.02)-- (2.9,0.94);
\draw (2.9,0.94)-- (1.96,1.84);
\draw (1.44,3.18)-- (4.48,4.98);
\draw (1.44,3.18)-- (5.7,3.94);
\draw (1.44,3.18)-- (5.4,1.76);
\draw (1.44,3.18)-- (4.56,1.02);
\draw (1.84,4.12)-- (1.96,1.84);
\draw (1.84,4.12)-- (5.7,3.94);
\draw (1.84,4.12)-- (4.48,4.98);
\draw (1.84,4.12)-- (2.9,0.94);
\draw (2.76,5.)-- (1.96,1.84);
\draw (2.76,5.)-- (2.9,0.94);
\draw (2.76,5.)-- (5.4,1.76);
\draw (2.76,5.)-- (5.8,2.88);
\draw (4.48,4.98)-- (5.4,1.76);
\draw (4.48,4.98)-- (5.8,2.88);
\draw (5.7,3.94)-- (4.56,1.02);
\draw (5.7,3.94)-- (2.9,0.94);
\draw (5.8,2.88)-- (2.9,0.94);
\draw (5.8,2.88)-- (4.56,1.02);
\draw (4.56,1.02)-- (1.96,1.84);
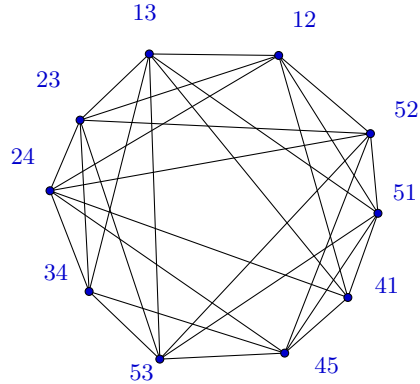
\draw (0.74,0.24) node[anchor=north west] {Figure1. The Johnson graph J(5,2)};
\begin{scriptsize}
\draw [fill=qqqqff] (1.44,3.18) circle (1.5pt);
\draw[color=qqqqff] (1.08,3.64) node {$24$};
\draw [fill=qqqqff] (1.84,4.12) circle (1.5pt);
\draw[color=qqqqff] (1.42,4.66) node {$23$};
\draw [fill=qqqqff] (2.76,5.) circle (1.5pt);
\draw[color=qqqqff] (2.7,5.56) node {$13$};
\draw [fill=qqqqff] (4.48,4.98) circle (1.5pt);
\draw[color=qqqqff] (4.82,5.46) node {$12$};
\draw [fill=qqqqff] (5.8,2.88) circle (1.5pt);
\draw[color=qqqqff] (6.16,3.16) node {$51$};
\draw [fill=qqqqff] (5.7,3.94) circle (1.5pt);
\draw[color=qqqqff] (6.18,4.26) node {$52$};
\draw [fill=qqqqff] (5.4,1.76) circle (1.5pt);
\draw[color=qqqqff] (5.92,1.92) node {$41$};
\draw [fill=qqqqff] (4.56,1.02) circle (1.5pt);
\draw[color=qqqqff] (5.12,0.86) node {$45$};
\draw [fill=qqqqff] (2.9,0.94) circle (1.5pt);
\draw[color=qqqqff] (2.66,0.76) node {$53$};
\draw [fill=qqqqff] (1.96,1.84) circle (1.5pt);
\draw[color=qqqqff] (1.52,2.1) node {$34$};
\end{scriptsize}
\end{tikzpicture} \

We know that complementation of subsets $ M \longmapsto  M^c$ induces an isomorphism $J(n,m) \cong    J(n, n-m
)$, hence we may assume
without loss of generality that  $m \leq \frac{n}{2}  $.
This graph has been studied by various authors and some of the recent papers are [1,4,6,8,9,14].
In this paper we determine the automorphism group  $Aut(J(n, m))$, for $6\leq n$ and $ m \leq \frac{n}{2}  $.
Actually, the automorphism group of $J(n, m)$ for both the $n = 2m$ and $n \neq 2m$ cases
was already determined in [8], but the proof given there uses heavy group-theoretic
machinery. The main result of [14] was to provide a proof for the   $n \neq 2m$ case that
uses only elementary group theory, the proof is based on an analysis of the clique
structure of the graph. In [14] the authors leave the $n = 2m$ case open but make a
conjecture for this case. Also in [6] the conjecture is  resolved  in the affirmative by providing
a proof that again uses only elementary group theory. We will again find these  results by different methods 
 which we belief are also elementary.\

\

\section{Preliminaries} 
In this paper, a graph $\Gamma=(V,E)$ is
considered as an undirected simple graph where $V=V(\Gamma)$ is the vertex-set
and $E=E(\Gamma)$ is the edge-set. For all the terminology and notation
not defined here, we follow $[3,7,14]$.

The graphs $\Gamma_1 = (V_1,E_1)$ and $\Gamma_2 =
(V_2,E_2)$ are called isomorphic, if there is a bijection $\alpha
: V_1 \longrightarrow V_2 $   such that  $\{a,b\} \in E_1$ if and
only if $\{\alpha(a),\alpha(b)\} \in E_2$ for all $a,b \in V_1$.
in such a case the bijection $\alpha$ is called an isomorphism.
An automorphism of a graph $\Gamma $ is an isomorphism of $\Gamma
$ with itself. The set of automorphisms of $\Gamma$  with the
operation of composition of functions is a group, called the
automorphism group of $\Gamma$ and denoted by $ Aut(\Gamma)$.
The
group of all permutations of a set $V$ is denoted by $Sym(V)$  or
just $Sym(n)$ when $|V| =n $. A permutation group $G$ on
$V$ is a subgroup of $Sym(V)$. In this case we say that $G$ act
on $V$. If $\Gamma$ is a graph with vertex-set $V$, then we can view
each automorphism as a permutation of $V$, and so $Aut(\Gamma)$ is a
permutation group. Let $G$ acts on $V$, we say that $G$ is
transitive (or $G$ acts transitively on $V$), if there is just
one orbit. This means that given any two elements $u$ and $v$ of
$V$, there is an element $ \beta $ of  $G$ such that  $\beta (u)= v
$.

The graph $\Gamma$ is called vertex transitive, if  $Aut(\Gamma)$
acts transitively on $V(\Gamma)$. The action of $Aut(\Gamma)$ on
$V(\Gamma)$ induces an action on $E(\Gamma)$, by the rule
$\beta\{x,y\}=\{\beta(x),\beta(y)\}$,  $\beta\in Aut(\Gamma)$, and
$\Gamma$ is called edge transitive if this action is
transitive. The graph $\Gamma$ is called symmetric, if  for all
vertices $u, v, x, y,$ of $\Gamma$ such that $u$ and $v$ are
adjacent, and $x$ and $y$ are adjacent, there is an automorphism
$\alpha$ such that $\alpha(u)=x$,   and $ \alpha(v)=y$. It is clear
that a symmetric graph is vertex transitive and edge transitive. The graph $\Gamma$ is called
 distance-transitive
if for all
vertices $u, v, x, y,$ of $\Gamma$ such that $ d(u,v) =d(x,y)$, there exists some $ g \in Aut(\Gamma)$ satisfying $g(u) = x $ and $ g(v) = y$. It is clear that a distance transitive graph is a symmetric graph. The Johnson graph $J(n,m)$ is an example
of a distance-transitive graph [3].
For $v\in V(\Gamma)$ and $G=Aut(\Gamma)$, the stabilizer subgroup
$G_v$ is the subgroup of $G$ containing of all automorphisms which
fix $v$. In the vertex transitive case all stabilizer subgroups
$G_v $ are conjugate in $G$, and consequently isomorphic, in this
case, the index of $G_v$ in $G$ is given by the equation,  $|G
: G_v| =\frac{|G|}{|G_v|} =|V(\Gamma)|
$. If each stabilizer $ G_v $ is the identity group, then no
element of $G$, except the identity,   fixes any vertex, and
we say that $G$ act semiregularly on $V$. We say that $G$ act
regularly on $V$ if and only if $G$ acts  transitively and
semiregularly on $ V$,
 and in this case we
have $|V| = |G|$.\

Although, 
In most situations  it is difficult  to determine the automorphism group
of a graph $\Gamma$ and how it acts on the vertex set of $\Gamma$,  there are various papers in the literature,   and some of the recent works  
appear in the references   [6,8,10,11,12,13,14,16].

\section{Main results}

Let $ \Gamma $ be a connected graph with diameter $ D $ and $ x$
be a vertex of $ \Gamma$. Let $ \Gamma_i= \Gamma_i(x) $ be the set of vertices of $ \Gamma$
at distance $i$ from $ x $. Thus $ \Gamma_0 = \{ x \}$,
and $ \Gamma_1 = N(x)$ is the set of vertices which are adjacent to the vertex $ x $. 
 Therefore,   $ V( \Gamma)$  is partitioned into the disjoint subsets $ \Gamma_0 (x),..., \Gamma_D (x)$.
Let $ v,w \in \Gamma $ and $ d(v,w)$ denotes the distance between the vertices $v$ and $w$
in the graph $ \Gamma$. It is an easy task to show that for any two vertices $v,w$ of  $J(n,m)$, $d(u,v) = k$ if and only if $|v \cap w|=m-k $ (when regarding $v,w$ as $m$-sets).

\begin{proposition}

Let $ \Gamma = J(n, m), \  n \geq 6, \  3 \leq m  \leq\frac{n}{2}  $. Let $ x \in V ( \Gamma ), \
 \Gamma_i = \Gamma_i (x) $ and $ v \in \Gamma_i $. Then we have; \
$$ \cap_{w \in \Gamma_{i-1} \cap N(v)} (N(w) \cap \Gamma_i) = \{ v \}$$

\end{proposition}

\begin{proof}
A proof of this result   appeared in [14],  but  for the sake of completeness and since  we
need our proof in the sequel, we offer a proof which is slightly different from that.
 It is clear that $ v \in \cap_{w \in \Gamma_{i-1} \cap N(v)} (N(w) \cap \Gamma_i)$.
Let $ x = \{ x_1, ..., x_m \} $. We can assume that  $ v = \{x_1, ..., x_{m-i}, y_1, ..., y_i \} $ where $ I = \{ x_1, ..., x_m, y_1, ...,y_{n-m} \} = \{ 1, 2, ..., n \} $. If $ w \in \Gamma_{i-1} \cap N (v) $, then $ | w \cap v | = m-1 $ and $ | w \cap x | = m - i + 1 $ and thus\
 $$ w = w_{rt} = ( v-y_t ) \cup \{ x_r \} = \{ x_1 , ..., x_{m-i}, x_r,  y_1, ..., y_{t-1}, y_{t+1}, ..., y_i \} $$  \ 
 where $ m-i+1 \leq r \leq m , \  1 \leq t \leq i $.\

 We show that if $ u \in \Gamma_i $ and $ u \neq v $ and $ u $ is adjacent to some $ w_{rt} $, then there  is some $ w_{pq} $ such that $ u $ is not adjacent to $ w_{pq} $.

 If $ v \neq u \in \Gamma_i $ is adjacent to $ w_{rt} $,  then $ u $ has one of the following forms,  $$ (*) \ \ \   \ u_1 = \{x_1, ..., x_{m-i}, y_1, ..., y_{t-1}, y_{t+1}, ..., y_i, y_a \}, \ i < a \leq n-m $$
  $$ (**) \ \ \   \  u_2 = \{ x_1, ..., x_{j-1}, x_{j+1}, ..., x_{m-i}, x_r, y_1, ..., y_{t-1}, y_{t+1}, ..., y_i, y_b\}$$ where
   $b\in \{i+1, ..., n-m \} \cup \{t\}, \ 1 \leq j \leq m-i$. \

 In the case $ (*) $, $ u_1 $ is not adjacent to $ w_{r(t-1)} $(or $ w_{r(t+1)}) $, because   $ x_r , y_t \in w_{r(t-1)} $ but $ x_r, y_t \not\in u_1 $. \

 In the case $ (**),  $ if $ y_b = y_t $, say, $$ u_2 = \{ x_1, ..., x_{j-1}, x_{j+1}, ..., x_{m-i}, x_r, y_1, ..., y_{t-1}, y_t, y_{t+1}, ..., y_i \} $$ then $ u_2 $ is not adjacent to $ w_{st} (s \neq r) $ because $ x_j, x_s \in w_{st} $ and $ x_j, x_s \not \in u_2 $.
 Also in the case $ (**) $, if $ y_b \neq y_t $, then we have $$ u_2 = \{ x_1, ..., x_{j-1}, x_{j+1}, ..., x_{m-i}, x_r, y_1, ..., y_{t-1}, y_{t+1}, ..., y_i, y_b \} $$ where $ b \in \{ i+1, ..., n-m \} $. But in this case $ u_2 $ is not adjacent to  $ w_{s(t-1)} $   (or   $ w_{s(t+1)}) $ because $ x_j, y_t \in w_{r(t-1)} $ and $ x_j, y_t \not\in u_2 $. \

 Our discussion shows that  if $ v \neq u \in \Gamma_i $,  then $ u \not\in \cap_{w \in N(v) \cap \Gamma_{i-1}}( N(w) \cap \Gamma_i) $ and thus we have $ \cap_{w \in N(v) \cap \Gamma_{i-1}}( N(w) \cap \Gamma_i)= \{v\} $.

\end{proof}

Let $ \Gamma $ be a graph, then the line graph $ L( \Gamma )$
of the graph $ \Gamma$ is constructed by taking the edges of $ \Gamma$
as vertices of  $ L( \Gamma )$, and joining two vertices in $ L( \Gamma )$
whenever the corresponding edges in $ \Gamma $ have a common vertex. There is
an important relation between $ Aut( \Gamma)$ and  $ Aut( L(\Gamma))$.  Indeed, we have
the following result [3, chapter 15] which obtained by Whitney [17].  Although,  the proof of this result is  tedious but uses elementary facts in  graph theory and group theory, and it is not difficult (see [2, chapter 13] for a proof).   \

\begin{theorem}
The mapping $ \theta: Aut( \Gamma) \rightarrow Aut( L(\Gamma))$ defined by;
$$ \theta(g) \{ u,v\}= \{  g(u), g(v)\}, \   g \in Aut( \Gamma), \ \{u,v \} \in E( \Gamma ) $$
is a group homomorphism and in fact we have;  \

$(i)$   $ \theta $ is a monomorphism provided  $ \Gamma \neq K_2 $; \

$(ii)$    $ \theta $ is an
epimorphism provided $\Gamma $ is not $K_4 $, $K_4 $ with one edge deleted, or $K_4 $ with two
edges deleted.
\end{theorem}\
For example, by using the above fact, we can obtain the following result.

\begin{proposition}
The automorphism group of Johnson graph $J(n,2)$ is isomorphic with the symmetric group $Sym(n)$. 

\end{proposition}
\begin{proof}
We know that the Johnson graph $J(n, 2)$ is the line graph of the complete graph on $n$ vertices, namely, $ J(n,2) \cong L(K_n)$. Thus, we have $Aut(J(n,2))$  $\cong Aut(L(K_n))$. By Theorem 3.2. it followes that  $Aut(L(K_n)) \cong Aut(K_n)$. Now, since  $Aut(K_n) \cong Sym(n)$, then we have $Aut(J(n,2)) \cong Sym(n)$.

\end{proof}
We now try to prove that the above result is true not only for the case $m=2$ in $J(n,m)$ but also   for every possible $m$.  

\begin{proposition}

Let $ v$ be a vertex of the Johnson graph $ J(n,m) $. Then,  $ \Gamma_1=<N(v)> $, the
induced subgraph of $ N(v)$ in $ J(n,m) $, is isomorphic
 with $ L(K_{m,n-m})$, where  $ K_{m,n-m} $ is the
  complete bipartite graph with  partitions of orders $ m$ and $ n-m$.

\end{proposition}

\begin{proof}
Let $ I = \{1,2,...,n \}, \  v= \{ x_1,...,x_m \}$
 and   $ w= v^c = \{ y_1,...,y_{n-m} \}$
be the complement of the subset $v$ in $ I$.
Let $ x_{ij}= v- \{x_j  \} \cup  \{ y_i \}$, $ 1\leq j  \leq m , \ 1 \leq i  \leq  n-m  $.
Then, \

 $$ N(v ) = \{ x_{ij} | 1 \leq j \leq m ,\ 1 \leq i \leq n-m \} $$
\

In $ \Gamma_1=<N(v)> $  two vertices $ x_{ij} $ $ x_{rs} $ are adjacent
if and only if $i=r $ or $ j= s $.  In fact,  $ \{  x_1, ..., x_{ j-1}, y_i, x_{ j+1},...,x_m \}$
and $ \{  x_1, ..., x_{ s-1}, y_r, x_{ s+1},...,x_m \}$ have $ m-1 $ element(s)
in common if and only if $x_i=x_r $ or $ y_j= y_s $.
\
Let $ X= \{ v_1,...,v_m \}  $ and $ Y= \{w_1,...,w_{n-m } \}  $ where $ X \cap Y = \varnothing $.
We know that
 the complete bipartite graph $ K_{ m, n-m} $
is the graph with vertex set $ X \cup  Y $,
 and edge set
 $ E= \{ \{ v_i, w_j \} , 1 \leq i\leq m ,\ 1 \leq j  \leq  n-m \} $.
 Then $  L ( K_{ n, n-m}) $ is the graph with vertex set
 $ V (L ( K_{ n, n-m}) )=E $ in which vertices $ \{ v_i,w_j \} $ and $ \{ v_r,w_s \} $
 are adjacent if and only if $v_i= v_r $ or $w_j=w_s$.
  Now it is an easy task to show that
 the mapping
\

\

\centerline{$ \phi : L ( K_{ n, n-m}) \rightarrow \Gamma_1=<N(v)> $, \
 $ \phi (v_i, w_j ) = x_{ ij} =  v- \{x_i  \} \cup  \{ y_j \} $}
\

  is a graph isomorphism.

\end{proof}

\begin{theorem}

Let $ \Gamma = J(n, m), \  n \geq 4, \  2 \leq m \leq  \frac {n}{2}  $. If $ n \neq 2m $,  then $ Aut (\Gamma )\cong Sym(n)$.
If $ n = 2m $,  then $ Aut(\Gamma) \cong Sym(n) \times\mathbb{Z}_2 $.

\end{theorem}

\begin{proof}Let $ G = Aut (\Gamma) $. Let $ x \in V = V( \Gamma )$, and $ G_x = \{f \in G | f(x) = x \}$ be the stabilizer subgroup of the vertex $ x $ in $ \Gamma $. Let $ <N(x)> = \Gamma_1 $ be the induced subgroup of
 $ N(x) $ in $ \Gamma $.  If $ f \in G_x $ then $ f _{|N(x)} $, the restriction of $ f $ to $ N(x) $ is an automorphism of $
 \Gamma_1 $. We define $ \varphi : G_x \longrightarrow Aut ( \Gamma_1) $ by the rule $ \varphi (f) = f_{ |N(x)} $. It is an easy task to show that $ \varphi $ is a group homomorphism. We show that $ ker (\varphi) $  is the identity group.
If $ f \in ker (\varphi) $, then $ f(x) = x $ and $ f(w) = w $ for  every $ w \in N(x) $. Let $ D $ be the diameter of $ \Gamma = J( n, m ) $ and $ \Gamma_i $ be the set of vertices of $ \Gamma $ at distance $ i $ from the vertex $ x$. Then $ V = V( \Gamma )= \cup_{i=0}^{D} \Gamma_i$.
We prove by induction on $ i $ that $ f(u) = u $ for every $ u \in \Gamma_i $. Let $ d(u, x ) $ be the distance of the vertex $ u $ from $ x $. If $ d(u, x) = 1 $,  then $ u \in \Gamma_1 $ and we have $ f(u) = u $. Assume $ f( u ) = u $ when $ d(u, x ) = i-1 $.
If $ d(u, x) = i $, then by  proposition 3.1.  $ \{ u \} = \cap_{w \in \Gamma_{i-1} \cap N(u)} N(w) $, and therefore $ f(u) =\cap_{w \in \Gamma_{i-1} \cap N(u)} N(f(w)) $.
Since, $ w \in \Gamma_{i-1} $, then   $ d(w, x) = i-1 $, and hence $f(w) = w $.
 Thus, \

\

\centerline{$ f(u) = \cap_{w \in \Gamma_{i-1} \cap N(u)} N(f(w)) = \cap_{w \in \Gamma_{i-1} \cap N(u)} N(w)=u $}

\

Thus,   $ ker (\varphi) $= $\{ 1 \} $.
On the order hand,
\

\centerline{$  \frac {G_v}{ker( \varphi)} \cong \varphi (G_v) \leq Aut (\Gamma_1) = Aut (<N(x)>)$
 and thus $ G_v \cong \varphi (G_v) \leq Aut (\Gamma_1 ) $}
\

 and hence  $ |G_v| \leq | Aut (\Gamma_1) | $. \

By Proposition 3.4.  $ \Gamma_1 \cong L( K_{m, n-m })$, thus $ Aut (\Gamma_1) \cong Aut (L( K_{m, n-m }))$. Since  by Theorem 3.2.  $Aut (L(K_{m, n-m}))$  $ \cong  $ $Aut(K_{m, n-m})$, hence  $ | Aut (\Gamma_1) |$
 = $ | Aut ( K_{m, n-m }) | $,  and therefore $ | G_v | \leq | Aut (K_{m, n-m}) |$.
  Note that if $ P = K_ { s,t}$  is a complete bipartite graph,  then  for $ t \neq s $  we have
   $ |Aut( P )|= s!t!$,  and for $ t=s $ we have $ |Aut( P )|= 2{(s!)}^2$ [3, chapter17].\

     Since $ \Gamma  = J(n,m)$ is a vertex transitive graph, we have $ |V(\Gamma)| = \frac { | G | }{ | G_v | }$, thus $ | G | = | G_v | | V (\Gamma) | \leq | Aut(K_{m, n-m})| {n \choose m } $.
Now, if $ n \neq 2m $, then  we have \

\

\centerline{$ | G | \leq ( m! )(n-m)! \frac {n!}{(m!) (n-m)!} = n! $}\

 and if  $ n=2m $, then we have\

\

\centerline{  (*) \ \ \ \  $ | G | \leq |Aut ( K_{m,m} )| {2m \choose m} $, and hence $ | G | \leq 2 (m!)^2 \frac {(2m)!}{(m!)(m!)} = 2(2m)! $} \

We know that
If $ \theta \in Sym (I) $ where $ I = \{1, 2, ..., n\} $,  then
\

\

\centerline{$ f_\theta : V( \Gamma )\longrightarrow V( \Gamma) $, $ f_\theta (\{x_1, ..., x_m \}) = \{ \theta (x_1), ..., \theta (x_m) \}$} \

 is an automorphism of $ \Gamma $ and the mapping  $ \psi : Sym (I) \longrightarrow Aut ( \Gamma )$, defined by the rule $ \psi ( \theta ) = f_\theta $ is an injection. \

Now if $ n \neq 2m$,
since  $ | G | = |Aut ( \Gamma ) | \leq n! $, we conclude that $ \psi $ is a bijection, and hence
 $ Aut ( \Gamma ) \cong Sym (I) $. \

If $ n = 2m $, then $ \Gamma = J (n, m) = J (2m, m) $ and the set  $ \{ f_\theta |\  \theta \in Sym (2m) \} = H $,  is a subgroup of $ Aut ( \Gamma ) $. It is an easy task to show that the  mapping  $ \alpha : V( \Gamma ) \longrightarrow V(\Gamma), \  \alpha( v ) = v^c $  where $ v^c $ is the complement of the set $ v $ in $I$, is
 also an automorphism of $ \Gamma$,  say,   $\alpha \in G = Aut (\Gamma ) $. \

We show that $ \alpha \not\in H $. If $ \alpha \in H $,  then there is a $ \theta \in Sym (I) $ such that $ f_\theta = \alpha $. Since $ o( \alpha ) = 2 $ $( o(\alpha )= $ order of $ \alpha ) $,  then $ o(f_\theta ) = o( \theta ) = 2 $. We assert that
$ \theta $ has no fixed points,  say,  $ \theta (x) \neq x $, for every  $ x \in I $. In fact,  if $ x \in I $, and $ \theta (x) = x $, then for the $ m $-set $ v = \{x, y_1, ..., y_{m-1} \} \subseteq I $, we have \

\

\centerline{$ f_{\theta }( v ) = \{ \theta (x), \theta (y_1), ..., \theta( y_{m-1}) \} = \{x, \theta(y_1), ..., \theta (y_{m-1}) \} $}
\

 hence  $ x \in f_\theta (v) \cap v $, and therefore  $ f_\theta (v) \neq v^c = \alpha (v) $  which is a contradiction.
Therefore,  $ \theta $ has a form such as   $ \theta = ( x_1, y_1 )...(x_m, y_m)$ where $ ( x_i, y_i)$ is a transposition of $Sym(I) $. Now, for the $m$-set $ v= \{x_1,y_1, x_2,...,x_{m-1}  \}$  we have
\

\

   \centerline{$ \alpha (v) = f_{ \theta }(v)= \{ \theta (x_1), \theta (y_1), \theta ( x_{m-1} )   \} $=
     $ \{ y_1, x_1,..., \theta (x_{m-1})   \} $} \

\

 and thus  $ x_1,y_1 \in f_ { \theta} (v) \cap v $,   hence $ f_\theta (v) \neq v^c = \alpha (v) $, which is a contradiction. \

We assert that  for every $\theta \in Sym(I)$, we have $f_{\theta}\alpha= \alpha f_{\theta}$. In fact, if $v=\{ x_1,...,x_m \}$ is a $m$-subset of $I$, then there are $y_j \in I,    1\leq j \leq m $,  such that $I= \{x_1,...,x_m, y_1,...,y_{m }   \}$. Now we have \

\

\centerline{ $f_{\theta}\alpha(v)=$
$  f_{\theta} \{ y_1,...,y_{m } \}= \{ \theta(y_1), ...,  \theta(y_{m })   \}$} \

 On the other hand, we have
\

\

\centerline{ $ \alpha f_{\theta}(v) =\alpha \{ \theta(x_1),..., \theta(x_m) \}= \{ \theta(y_1), ...,  \theta(y_{m})   \}$} \

because $I = \theta(I)= \{\theta(x_1),...,\theta(x_m), \theta(y_1),...,\theta(y_{m} )        \}$.  Consequently,   $f_{\theta}\alpha(v)=  \alpha f_{\theta}(v)  $.   We now deduce that $f_{\theta}\alpha=  \alpha f_{\theta}  $. \

Note that if $X$ is a group and $ Y,Z$ are subgroups of $X$, then  the subset $YZ=\{  yz \ |  \  y\in Y , z\in Z \}$ is a subgroup in X if and only if $YZ=ZY$.  According to this fact,    we conclude that  $ H  < \alpha > $ is a subgroup of $G$.

         Since $ \alpha \not\in H $ and $ o( \alpha ) =2$,  then $ H < \alpha > $ is a subgroup of $ G $ of order\

\

\centerline{ $ \frac { |H| |< \alpha > |}{ | H \cap < \alpha > |} = 2|H| = 2( ( 2m ) )! $ } \

Now,  since  by (*) $ | G | \leq 2( ( 2m )! ) $,  then $ G =  H  < \alpha > $. On the other hand,  since $f_{\theta}\alpha=  \alpha f_{\theta}  $, for every $ \theta \in Sym(I) $, then  $ H $ and
$< \alpha >$ are normal subgroups of $ G $. Thus,   $ G $    is a direct product of two  groups $H$ and $< \alpha >$, 
  namely, we have  $ G = H \times < \alpha > \cong Sym (2m)\times \mathbb{Z}_2 $.

\end{proof}

\

\end{document}